\documentclass{amsart}
\usepackage{times,amssymb,amscd,verbatim,color,graphics}
\usepackage[vcentermath,stdtext,enableskew]{youngtab}
\usepackage[all]{xy}

\numberwithin{equation}{section}

\newtheorem{prop}{Proposition}
\newtheorem{theorem}[prop]{Theorem}
\newtheorem{corollary}[prop]{Corollary}
\newtheorem{lemma}[prop]{Lemma}
\newtheorem{fact}[prop]{Fact}

\newtheorem{problem}[prop]{Problem}

\theoremstyle{definition}
\newtheorem{definition}[prop]{Definition}
\newtheorem{example}[prop]{Example}
\newtheorem{remark}[prop]{Remark}

\numberwithin{prop}{section}

\newcommand{\maj}{\mathrm{maj}}
\newcommand{\comaj}{\mathrm{comaj}}
\newcommand{\Des}{\mathrm{Des}}
\newcommand{\extDes}{\mathrm{Des_e}}
\newcommand{\barmaj}{\overline{\mathrm{maj}}}
\newcommand{\barcomaj}{\overline{\mathrm{comaj}}}
\newcommand{\bnumber}{\mathrm{b}}

\newcommand{\ten}{10}
\newcommand{\eleven}{11}
\newcommand{\twelv}{12}
\newcommand{\thirteen}{13}
\newcommand{\fourteen}{14}
\newcommand{\fifteen}{15}
\newcommand{\sixteen}{16}

\newcommand{\bone}{\textcolor{blue}{1}}
\newcommand{\btwo}{\textcolor{blue}{2}}
\newcommand{\bthree}{\textcolor{blue}{3}}

\newcommand{\bseven}{\textcolor{blue}{7}}
\newcommand{\beight}{\textcolor{blue}{8}}
\newcommand{\bnine}{\textcolor{blue}{9}}
\newcommand{\bten}{\textcolor{blue}{10}}

\newcommand{\bfifteen}{\textcolor{blue}{15}}

\newcommand{\boldten}{\textcolor{red}{\mathbf{10}}}
\newcommand{\boldnine}{\mathbf{\textcolor{red}{9}}}
\newcommand{\boldeight}{\mathbf{\textcolor{red}{8}}}
\newcommand{\boldseven}{\mathbf{\textcolor{red}{7}}}
\newcommand{\boldsix}{\mathbf{\textcolor{red}{6}}}
\newcommand{\boldfive}{\mathbf{\textcolor{red}{5}}}
\newcommand{\boldfour}{\mathbf{\textcolor{red}{4}}}
\newcommand{\boldthree}{\mathbf{\textcolor{red}{3}}}
\newcommand{\boldtwo}{\mathbf{\textcolor{red}{2}}}
\newcommand{\boldone}{\mathbf{\textcolor{red}{1}}}

\begin{document}
\title{Promotion and evacuation on standard Young tableaux of rectangle and staircase shape}

\author[S.~Pon]{Steven Pon}
\address{Department of Mathematics, University of California, One Shields
Avenue, Davis, CA 95616-8633, U.S.A.}
\email{spon@math.ucdavis.edu}
\urladdr{http://www.math.ucdavis.edu/\~{}spon}

\author[Q.~Wang]{Qiang Wang}
\email{xqwang@math.ucdavis.edu}
\urladdr{http://www.math.ucdavis.edu/\~{}xqwang}
%\thanks{\textit{Date:} August 2009}
\thanks{Both authors were partially supported by NSF grants DMS--0652641, and DMS--0652652 
for this work}

\begin{abstract}
(Dual-)promotion and (dual-)evacuation are bijections on $SYT(\lambda)$ for any partition 
$\lambda$. Let $c^r$ denote the rectangular partition $(c,\ldots,c)$ of height $r$, 
and let $sc_k$ ($k > 2$) denote the staircase partition $(k,k-1,\ldots,1)$.  
B. Rhoades showed representation-theoretically that promotion on $SYT(c^r)$ exhibits 
the \textit{cyclic sieving phenomenon} (CSP).
In this paper, we demonstrate a promotion- and evacuation-preserving embedding 
of $SYT(sc_k)$ into $SYT(k^{k+1})$.  This arose from an attempt to demonstrate
 the CSP of promotion action on $SYT(sc_k)$.

%\wqsays{We do not have a clear idea how this embedding can be helpful in showing the CSP of
%pormotion action on $SYT(sc_k)$, nor do we know if this embedding has any 
%representation-theoretical interpretation.} 

\end{abstract}

\maketitle

%%%%%%%%%%%%%%%%%%%%%%%%%%%%%%%%%%%%%%%%%%%%%%%%%%%%%%%%%%%%%%%%%%%%%%%%%%%%%%%%%%%%%%%%%%%%%
%Section 1
%%%%%%%%%%%%%%%%%%%%%%%%%%%%%%%%%%%%%%%%%%%%%%%%%%%%%%%%%%%%%%%%%%%%%%%%%%%%%%%%%%%%%%%%%%%%%
\section{Introduction}

Let $X$ be a finite set and let $C=\langle a \rangle$ be a cyclic group of order $N$
acting on $X$. Let $X(q) \in \mathbb{Z}[q]$ be a polynomial with integer coefficients.
We say that the triple $(X,C,X(q))$ exhibits the \textbf{cyclic sieving phenomenon} (CSP) 
if for any integer $k$, we have 
\begin {equation} \label {eq: CSP polynomial}
    X(\zeta^k)  = \#\{x\in X \mid a^k \cdotp x=x\},
\end {equation}
where $\zeta = e^{2\pi i/N}$ is a $N$-th root of unity. We will call $X(q)$ a 
\textbf{CSP polynomial}.

Given $X$ and $C$, the existence of a CSP polynomial with non-negative integer coefficients
is necessary. Indeed, if we view $a \in Perm(X)$ and let $m_c$ be the multiplicity of cycle(s) 
of size $c$ in the cycle notation of $a$ (clearly, $m_c > 0$ implies that $c | N$), 
and let $p_c = \sum_{k=0,\ldots,c-1} q^{k\cdots N/c}$, then $$ p_{a,X}=\sum_{c \in \mathbb{N}}m_c\cdots p_c $$
is such a CSP polynomial. 
 
The set of all CSP polynomials forms a coset $\overline{p_{a,X}}$ in the ring 
$\mathbb{Z}[q]/\langle q^N-1 \rangle$ and $p_{a,X}$ is the least degree representative of 
this coset.

The interesting instances of the CSP (\cite{r-s-w}, \cite{k-m}, \cite{westbury},
\cite{s-s-w}, \cite{p-s}, \cite{p-p-r}, \cite{be-re}, and \cite{e-f}, etc.) 
are those whose CSP polynomials have a ``natural'' meaning,
for example, the $q$-analogues of some counting formulae of the set $X$. In nearly all of
these interesting instances of the CSP, $X(q)$ is also a generating function 
$$X(q) = \sum_{x\in X} q^{\mu(x)}$$ 
of an intrinsic statistic $\mu : X \to \mathbb{Z}$ on $X$.  
In the CSP instances where such an intrinsic statistic exists, we use the more
explicit triple $(X,C,\mu)$ to denote it.

\begin{comment} 
Let $S$ be a finite set and let $a \in Perm(S)$ with $N=|a|$;
let $m_C$ be the multiplicity of cycle(s) of size $C$ in the cycle notation of $a$ 
(clearly, $m_C > 0$ implies that $C | N$); 
let $p_C = \sum_{k=0,\ldots,C-1} x^{k\cdots N/C}$; and
let $\zeta = e^{2\pi i/N}$. Then $$ p_{a,S}=\sum_{C \in \mathbb{N}}m_C\cdots p_C $$
is a polynomial with non-negative integer coefficients that satisfies the condition: 
\begin {equation} \label {eq: CSP polynomial}
    p(\zeta^k)  = \#\{t\in S \mid a^k \cdotp t=t\}.
\end {equation}
for each $k \in \{0,\ldots,N-1\}$.  Conversely, given the number of fixed points of $a^k$ for each $k$, the cycle structure 
of $a$ is completely determined.
 
The set of all the integer coefficient polynomials that satisfy 
\eqref{eq: CSP polynomial} forms a coset $\overline{p_{a,S}}$ in the ring 
$\mathbb{Z}[x]/\langle x^N-1 \rangle$ and $p_{a,S}$ is the least degree representative of 
this coset.

If there is a statistic $\mu : S \to \mathbb{Z}$ on $S$ with the property that
the generating function of $\mu$ is congruent to $p_{a,S}$, that is  
$\sum_{t\in S} x^{\mu(t)} \in \overline{p_{a,S}}$, then $(a, S, \mu)$ is said to exhibit the
\textbf{cyclic sieving phenomenon} (CSP), and any element in $\overline{p_{a,S}}$ is
called a \textbf{CSP polynomial}.
\end{comment}

V. Reiner, D. Stanton, and D. White first formalized the notion of the CSP in \cite{r-s-w}.
Before them, Stembridge considered the ``$q=-1$'' phenomenon \cite{stemb}, 
which is the special case of the CSP with $N=2$ (where $\zeta=e^{2\pi i/2}=-1$). 

Promotion $\partial$ (Definition~\ref{definition: promotion}) and 
evacuation $\epsilon$ (Definition~\ref{definition: evacuation}) are closely related 
permutations on the set of standard Young tableaux $SYT(\lambda)$ for any given shape 
$\lambda$. Sch\"utzenberger studied them in \cite{schutz1, schutz2, schutz3}
as bijections on $SYT(\lambda)$, and later as permutations
on the linear extensions of any finite poset. 
Edelman and Greene \cite{e-g}, and Haiman~\cite{haiman} showed important properties of them; 
in particular, they found the order of promotion on $SYT(\lambda)$ when $\lambda$ is a 
staircase shape. In 2008, Stanley gave a terrific survey \cite{RS:2008} of previous results on
promotion and evacuation.

Important instances of the CSP arise from the actions of promotion and evacuation on 
standard Young tableaux. For example, Stembridge \cite{stemb} showed that 
$(SYT(\lambda), \langle\epsilon\rangle, \overline{\comaj})$ exhibits the CSP, 
%that is,    
%$$\sum_{w\in SYT(\lambda)}(-1)^{\overline{\comaj}(w)}=\#\{t\in SYT(\lambda)\mid \epsilon(t)=t\},$$
where $\lambda$ is any partition shape, and 
$\overline{comaj}$ is a statistic
on standard Young tableaux that is closely related to the comajor index. 

As another example, 
B. Rhoades \cite{rhoades} showed that $(SYT(c^r),\langle\partial\rangle, \overline{\maj})$ 
exhibits the CSP, 
where $c^r$ is the rectangular partition of $r$ equal parts of size $c$, 
$\partial$ is promotion  on $SYT(c^r)$, and
$\overline{\maj}$ is a statistic on 
standard Young tableaux that is closely related to the major index. 

%Both Stembridge and Rhoades' proof are representation-theoretically. The main problems 
%of ours are:
%\begin{problem} \label{problems 1}
%Find combinatorial proofs of Stembridge and Rhoades' results. 
%\end{problem}
%And
%\begin{problem} \label{problems 2}
%Demonstrate CSP of promotion action on staircase tableaux $SYT(sc_k=(k,k-1,\ldots,1)$.  
%\end{problem}

Since the introduction of the CSP in \cite{r-s-w}, much effort has been made in 
demonstrating interesting instances of it (\cite{k-m}, \cite{westbury},
\cite{s-s-w}, \cite{p-s}, \cite{p-p-r}, \cite{be-re}, and \cite{e-f}, etc.), 
or generalizing it (\cite{b-e-r-s}). 

In this paper we report our current progress in attacking the 
following problems:
\begin{problem} \label{problem 2}
Demonstrate the CSP of promotion action on staircase tableaux $SYT(sc_k=(k,k-1,\ldots,1))$.  
More specifically, this could mean any one of the following three tasks, listed by increasing difficulty:
\begin{itemize}
  \item Find a counting formula for $SYT(sc_k)$, the $q$-analogue of which provides
        a CSP polynomial for the promotion action on $SYT(sc_k)$.
  \item Find a ``natural'' statistic on $SYT(sc_k)$, the generating function of which provides
        a CSP polynomial for the promotion action on $SYT(sc_k)$.
  \item Find a statistic on $SYT(\lambda)$, the generating function of which provides
        a CSP polynomial for the promotion action on $SYT(\lambda)$. In particular,
        this statistic should have the same distribution as $\overline{\maj}$ on
        $SYT(c^r)$.
\end{itemize}
\end{problem}

The reported progress is the
construction of an embedding $\iota:  SYT(sc_k) \hookrightarrow SYT(k^{(k+1)})$ that 
preserves promotion and evacuation. This enables us to extend Rhoades' definition 
of the ``extended descent'' from rectangular tableaux to staircase tableaux.  

%characterize (\TODO{still a
%conjecture}) all the possible cycles appear in the promotion action on $SYT(sc_k)$.

This paper is organized in the following way:   
In Section~\ref{section: definitions and preliminaries}, we define the terminology and 
notation, and review several basic results that are used in later sections. 

In Section~\ref{section: the embedding}, we construct the embedding $\iota$ and prove
our main results about $\iota$: Theorems \ref{theorem: embedding preserves promotion} 
and \ref{theorem: embedding preserves evacuation}, which state that promotion and evacuation are
preserved under the embedding.

In Section~\ref{section: descent vector}, we extend Rhoades' construction of ``extended descent''
on rectangular tableaux to that on staircase tableaux by using $\iota$; 
our main results in this section are Theorems \ref{theorem: edv and iota}, \ref{theorem: posdv},
 \ref{theorem: eosdv} and \ref{theorem: deosdv}, which state that the extended descent 
data nicely records the actions of (dual-)promotion and (dual-)evacuation on both rectangular
and staircase tableaux.

In Section~\ref{section: future}, we explain how the embedding $\iota$ arose and pose 
some questions about it. 

%Finally, in Appendix~\ref{appendix: random facts}, we document several facts about 
%tableaux we ran into as the by-product of this investigation; and 
%in Appendix~\ref{appendix: proposal}, we document some of our ideas on attacking    
%Problem~\ref{promblem 1} and \ref{problem 2}.

%\subsection*{Acknowledgements}

%%%%%%%%%%%%%%%%%%%%%%%%%%%%%%%%%%%%%%%%%%%%%%%%%%%%%%%%%%%%%%%%%%%%%%%%%%%%%%%%%%%%%%%%%%%%%
%Section 2
%%%%%%%%%%%%%%%%%%%%%%%%%%%%%%%%%%%%%%%%%%%%%%%%%%%%%%%%%%%%%%%%%%%%%%%%%%%%%%%%%%%%%%%%%%%%%
\section{Definitions and Preliminaries} \label{section: definitions and preliminaries}

This section is a review of those notions, notations and facts about Young tableaux that
are directly used in the following sections. 
We assume the reader's basic knowledge of tableaux theory -- partitions, standard Young 
tableaux, Knuth equivalence, reading word of a tableau, jeu-de-taquin, and RSK algorithm, etc.  All
of our tableaux and directional references (e.g., north, west, etc.) will refer to tableaux
in ``English'' notation.  For more on these topics, see \cite{ec2} or \cite{ful}.
\subsection{Basic definitions}
  
\begin{definition}\label{definition: promotion}
Given $T \in SYT(\lambda)$ for any (skew) shape $\lambda \vdash n$, the \textbf{promotion} action 
on $T$, denoted by $\partial(T)$, is given as follows:

Find in $T$ the outside corner that contains the number $n$, and remove it to create an empty box.
Apply jeu-de-taquin repeatedly to move the empty box northwest until the empty box is
an inside corner of $\lambda$. (We call this process \textbf{sliding}, the sequence of
positions that the empty box moves along in this process is called \textbf{sliding path}).   
Place $0$ in the empty box. Now add one to each entry of the current filling of $\lambda$ 
so that we again have a standard Young tableau.  This new tableau is $\partial(T)$,
the promotion of $T$. 

In the case that sliding is used to define promotion, we will refer to the sliding path as the
\textbf{promotion path}. 
\end{definition}

\begin{remark}
Edelman and Greene (\cite{e-g}) call $\partial$ defined above ``elementary promotion.''
They call $\partial^n$ the ``promotion operator.''

%Our definition of $\partial$ would be called ``elementary promotion'' by Edelman and
%Greene in \cite{e-g}. They would use the name ``promotion operator'' for $\partial^n$.
\end{remark}

\begin{example}{Promotion on standard tableaux.}
$$
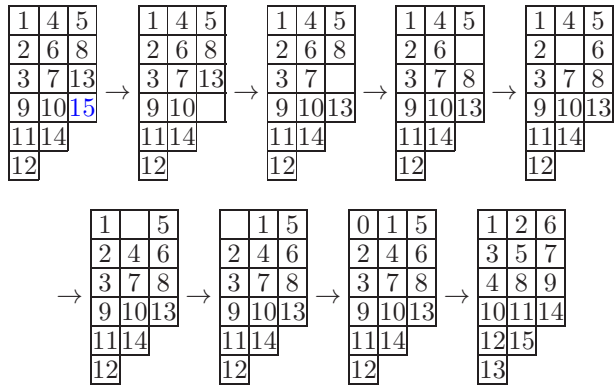
\young(145,268,37\thirteen,9\ten\bfifteen,\eleven\fourteen,\twelv)
\to
\young(145,268,37\thirteen,9\ten\hfil,\eleven\fourteen,\twelv)
\to
\young(145,268,37\hfil,9\ten\thirteen,\eleven\fourteen,\twelv)
\to
\young(145,26\hfil,378,9\ten\thirteen,\eleven\fourteen,\twelv)
\to
\young(145,2\hfil6,378,9\ten\thirteen,\eleven\fourteen,\twelv)
$$
$$
\to
\young(1\hfil5,246,378,9\ten\thirteen,\eleven\fourteen,\twelv)
\to
\young(\hfil15,246,378,9\ten\thirteen,\eleven\fourteen,\twelv)
\to
\young(015,246,378,9\ten\thirteen,\eleven\fourteen,\twelv)
\to
\young(126,357,489,\ten\eleven\fourteen,\twelv\fifteen,\thirteen)
$$

If we label the boxes by $(i,j)$, with $i$ being the row index from top to bottom and $j$
being the column index from left to right, and the northwest corner being labelled $(1,1)$, 
then the promotion path corresponding to the above example is 
$[(4,3),(3,3),(2,3),(2,2),(1,2),(1,1)]$.
\end{example}

Promotion $\partial$ has a dual operation, called \textbf{dual-promotion}, denoted 
by $\partial^*$ and defined as follows:

\begin{definition}
Find in $T$ the inside corner that contains $1$, and remove it to create 
an empty box.  
Apply jeu-de-taquin repeatedly to move the empty box southeast until it is an outside corner of 
$\lambda$. (We call this process \textbf{dual-sliding}, and the sequence of
positions that the empty box moves along in this process is called the \textbf{dual-sliding path}). 
Place the number $n+1$ in this outside corner. Now subtract one from each entry so that 
we again have a standard Young tableau.  
This new tableau is $\partial^*(T)$, the dual-promotion of $T$.

In the case that dual-sliding is used to define promotion, we will refer to the dual-sliding 
path as the \textbf{dual-promotion path}. 
\end{definition}

\begin{example}{Dual-promotion on standard tableaux.}

$$
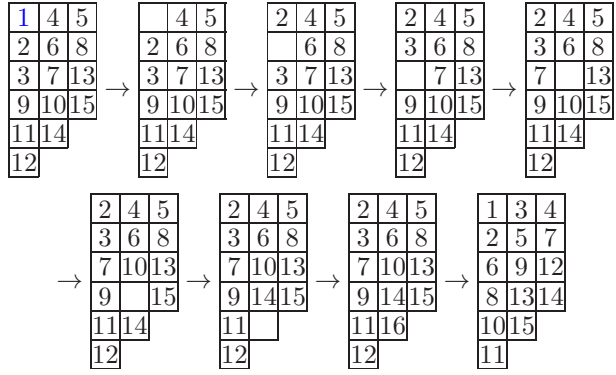
\young(\bone45,268,37\thirteen,9\ten\fifteen,\eleven\fourteen,\twelv)
\to
\young(\hfil45,268,37\thirteen,9\ten\fifteen,\eleven\fourteen,\twelv)
\to
\young(245,\hfil68,37\thirteen,9\ten\fifteen,\eleven\fourteen,\twelv)
\to
\young(245,368,\hfil7\thirteen,9\ten\fifteen,\eleven\fourteen,\twelv)
\to
\young(245,368,7\hfil\thirteen,9\ten\fifteen,\eleven\fourteen,\twelv)
$$
$$
\to
\young(245,368,7\ten\thirteen,9\hfil\fifteen,\eleven\fourteen,\twelv)
\to
\young(245,368,7\ten\thirteen,9\fourteen\fifteen,\eleven\hfil,\twelv)
\to
\young(245,368,7\ten\thirteen,9\fourteen\fifteen,\eleven\sixteen,\twelv)
\to
\young(134,257,69\twelv,8\thirteen\fourteen,\ten\fifteen,\eleven)
$$

The dual-promotion path of the above example is $[ (1,1), (2,1), (3,1), (3,2), (4,2), (5,2) ]$.

\end{example}

\begin{remark} \label{remark: promotion and dual-promotion path}
It is easy to see that $\partial^*=\partial^{-1}$; thus, they are both bijections
on $SYT(\lambda)$. Moreover, the the promotion path of $T$ is the reverse of the 
dual-promotion path of $\partial(T)$. 
\end{remark}

\begin{definition}\label{definition: evacuation}

Given $T \in SYT(\lambda)$ for any $\lambda \vdash n$, the \textbf{evacuation} action 
on $T$, denoted by $\epsilon(T)$, is described in the following algorithm:

Let $T_0=T$ and $\lambda_0=\lambda$, and let $U$ be an ``empty'' tableau of shape $\lambda$.  
We will fill in the entries of $U$ to get $\epsilon(T)$.

\begin{enumerate}
\item Apply sliding to $T_k$.  The last box of the sliding path is an inside corner of
$\lambda_k$; call this box $(i_k,j_k)$.  Fill in the number $k$ in the $(i_k,j_k)$ box of $U$.

\item Remove $(i_k,j_k)$ from $\lambda_k$ to get $\lambda_{k+1}$, and remove the corresponding box
and entry from $T_k$ to get $T_{k+1}$.

\item Repeat steps (1) and (2) $n$ times until $\lambda_n = \emptyset$ and $U$ is completely filled. 
      Then define $\epsilon(T)=U$.

\end{enumerate}
\end{definition}

\begin{example}
The following is a ``slow motion'' demonstration of the above process, 
where the $T_k$ and $U$ have been condensed.  Bold entries indicate the current fillings of $U$.
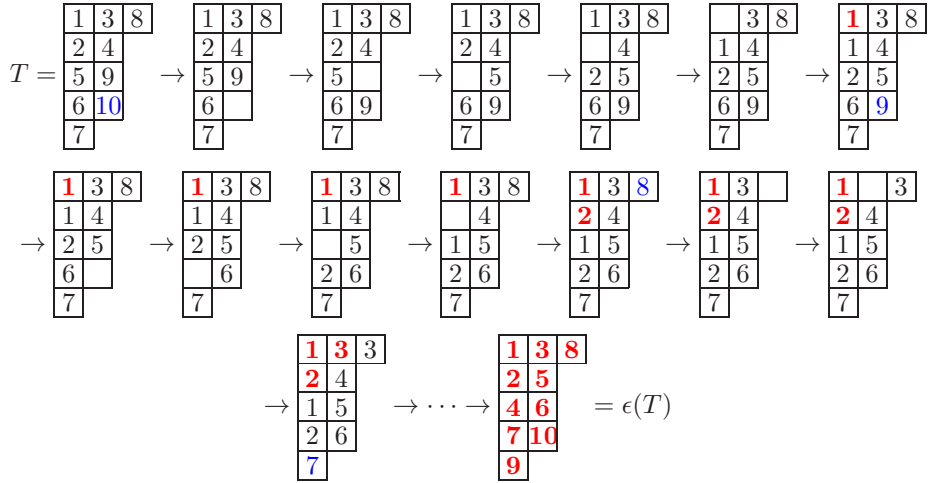
$$
   T =  \young(138,24,59,6\bten,7)
\to
    \young(138,24,59,6\hfil,7)
\to
    \young(138,24,5\hfil,69,7)
\to
    \young(138,24,\hfil5,69,7)
\to
    \young(138,\hfil4,25,69,7)
\to
    \young(\hfil38,14,25,69,7)
\to
    \young(\boldone38,14,25,6\bnine,7)
$$
$$
\to
    \young(\boldone38,14,25,6\hfil,7)
\to
    \young(\boldone38,14,25,\hfil6,7)
\to
    \young(\boldone38,14,\hfil5,26,7)
\to
    \young(\boldone38,\hfil4,15,26,7)
\to
    \young(\boldone3\beight,\boldtwo4,15,26,7)
\to
    \young(\boldone3\hfil,\boldtwo4,15,26,7)
\to
    \young(\boldone\hfil3,\boldtwo4,15,26,7)
$$
$$
\to
    \young(\boldone\boldthree3,\boldtwo4,15,26,\bseven)
\to
\cdots
\to
    \young(\boldone\boldthree\boldeight,\boldtwo\boldfive,\boldfour\boldsix,\boldseven\boldten,\boldnine) = \epsilon(T)
$$
\end{example}

\begin{remark} \label{remark: e-g convention}
The above definition of evacuation follows the convention of Edelman and Greene in \cite{e-g}. 
Stanley's ``evacuation'' \cite[A1.2.8]{ec2} would be our ``dual-evacuation'' defined below.  
\end{remark}

\begin{definition}\label{definition: dual-evacuation}
Given $T \in SYT(\lambda)$ for any $\lambda \vdash n$, the \textbf{dual-evacuation} of $T$, 
denoted by $\epsilon^*(T)$, is described in the following algorithm:

Let $T_0=T$ and $\lambda_0=\lambda$, and let $U$ be an ``empty'' tableau of shape $\lambda$.  
We will fill in the entries of $U$ to get $\epsilon^*(T)$.

\begin{enumerate}
\item Apply dual-sliding to $T_k$.  The last box of the dual-sliding path is an 
outside corner of $\lambda_k$; call this box $(i_k,j_k)$.  
Fill in the number $n+1-k$ in the $(i_k,j_k)$ box of $U$.

\item Remove $(i_k,j_k)$ from $\lambda_k$ to get $\lambda_{k+1}$, and remove the corresponding box
and entry from $T_k$ to get $T_{k+1}$.

\item Repeat steps (1) and (2) $n$ times until 
      $\lambda_n = \emptyset$ and $U$ is completely filled. 
      Then define $\epsilon^*(T)=U$.

\end{enumerate}
\end{definition}

\begin{example}
The following is a ``slow motion'' demonstration of the above process, 
where the $T_k$ and $U$ have been condensed.  Bold entries indicate the current fillings of $U$.
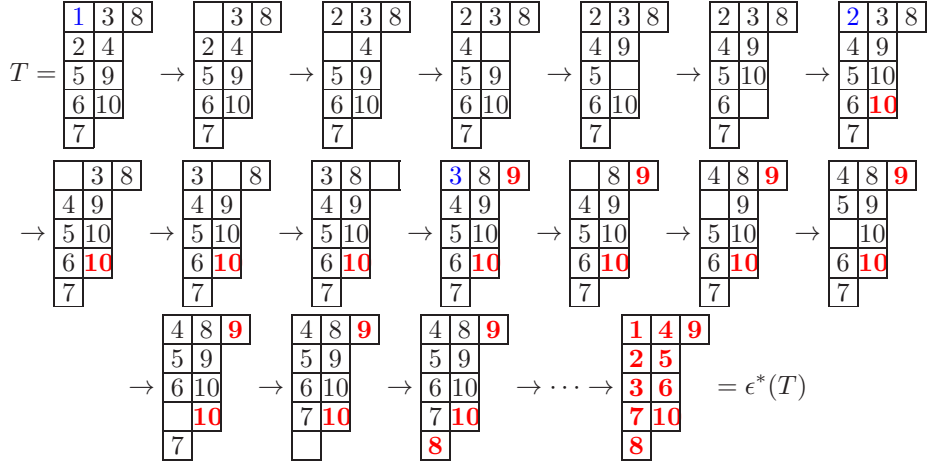
$$
   T =  \young(\bone38,24,59,6\ten,7)
\to
    \young(\hfil38,24,59,6\ten,7)
\to
    \young(238,\hfil4,59,6\ten,7)
\to
    \young(238,4\hfil,59,6\ten,7)
\to
    \young(238,49,5\hfil,6\ten,7)
\to
    \young(238,49,5\ten,6\hfil,7)
\to
    \young(\btwo38,49,5\ten,6\boldten,7)
$$
$$
\to
    \young(\hfil38,49,5\ten,6\boldten,7)
\to
    \young(3\hfil8,49,5\ten,6\boldten,7)
\to
    \young(38\hfil,49,5\ten,6\boldten,7)
\to
    \young(\bthree8\boldnine,49,5\ten,6\boldten,7)
\to
    \young(\hfil8\boldnine,49,5\ten,6\boldten,7)
\to
    \young(48\boldnine,\hfil9,5\ten,6\boldten,7)
\to
    \young(48\boldnine,59,\hfil\ten,6\boldten,7)
$$
$$
\to
    \young(48\boldnine,59,6\ten,\hfil\boldten,7)
\to
    \young(48\boldnine,59,6\ten,7\boldten,\hfil)
\to
    \young(48\boldnine,59,6\ten,7\boldten,\boldeight)
\to
\cdots
\to
    \young(\boldone\boldfour\boldnine,\boldtwo\boldfive,\boldthree\boldsix,\boldseven\boldten,\boldeight) = \epsilon^*(T)
$$
%$$
%\to
%    \young(\hfil38,49,5\ten,6\rten,7)
%\to
%    \young(3\hfil8,49,5\ten,6\rten,7)
%\to
%    \young(38\hfil,49,5\ten,6\rten,7)
%\to
%    \young(\bthree8\rnine,49,5\ten,6\rten,7)
%\to
%    \young(\hfil8\rnine,49,5\ten,6\rten,7)
%\to
%    \young(48\rnine,\hfil9,5\ten,6\rten,7)
%\to
%    \young(48\rnine,59,\hfil\ten,6\rten,7)
%$$
%$$
%\to
%    \young(48\rnine,59,6\ten,\hfil\rten,7)
%\to
%    \young(48\rnine,59,6\ten,\hfil\rten,7)
%\to
%    \young(\bfour8\rnine,59,6\ten,7\rten,\reight)
%\to
%\cdots
%\to
%    \young(\rone\rthree\reight,\rtwo\rfive,\rfour\rsix,\rseven\rten,\rnine)
%$$
\end{example}

\begin{remark} \label{remark: RSK evacuation}
There is an equivalent definition of $\epsilon^*$ via the RSK algorithm \cite[A1.2.10]{ec2}.
(Recall that Stanley's ``evacuation'' is our ``dual-evacuation''.)
For a permutation $w=w_1w_2\cdots w_n \in \mathfrak{S}_n$ (in one-line notation), let 
$w^{\sharp} \in \mathfrak{S}_n$ be given by $$w^{\sharp}=(n+1-w_n)\cdots(n+1-w_2)(n+1-w_1).$$ 
For example, in the case $w=3547126$, $w^{\sharp}=2671435$.  The operation $w \to w^\sharp$ 
is equivalent to composing by the longest element in $\mathfrak{S}_n$.
Then if $w$ corresponds to $(P,Q)$ under RSK, $w^{\sharp}$ corresponds to 
$(\epsilon^*(P),\epsilon^*(Q))$ under RSK.
We are not aware of any RSK definition of $\epsilon$ for general shape $\lambda$.
\end{remark}

%\begin{remark} \label{remark: evacuation is a sequence is dual promotion}
%It is clear from the definition that evacuation $\epsilon$ can be viewed as
%a consecutive application of dual-promotions on the resulting ``black'' portion of the  
%of the tableau. If we apply promotion first then and then evacuation, the first 
%dual-promotion in the evacuation will cancel the promotion. In particular, let $n$
%be the largest number in a tableau $S$, then the position of $n$ in 
%$\epsilon\circ\partial(S)$ is the same as that of in $S$.
%\end{remark}

\begin{definition}\label{definition: descent}
For $T \in SYT(\lambda)$, $i$ is a \textbf{descent} of 
$T$ if $i+1$ appears strictly south of $i$ in $T$. The \textbf{descent set} of $T$,
denoted by $\Des(T)$, is the set of all descents of $T$. 
\end{definition}

\begin{example}
In the case that $T=\young(123,469,57,8)$, $\Des(T)=\{3,4,6,7\}$. 
\end{example}

\begin{remark} \label{remark: reading word descent}
Descent statistics were originally defined on permutations. For $\pi \in \mathfrak{S}_n$,
$i$ is a \textbf{right descent} of $\pi$ if $\pi(i)>\pi(i+1)$, and $i$ is a 
\textbf{left descent} of $\pi$ if $i$ is to the right of $i+1$ in the one-line notation
of $\pi$.

It is straightforward to check that left descents are preserved by Knuth equivalence. 
Therefore the descent set of any tableau $T$ is the set of left descents of any reading word 
of $T$.
\end{remark}

\begin{comment}

\begin{definition}\label{definition: maj}
The \textbf{major index} of $T \in SYT(\lambda)$, denoted by $\maj(T)$, is the sum of all
descents of $T$: $$ \sum_{d\in \Des(T)}d .$$ 
The \textbf{comajor index} of $T \in SYT(\lambda)$, denoted by $\comaj(T)$, is the sum: $$ \sum_{d\in \Des(T)} (|\lambda|-d).$$   
\end{definition}

\begin{remark}
It is well known that $\maj(T) = \comaj(\epsilon(T))$. 
\end{remark}

\begin{definition}
Give a shape $\lambda$, we define the \textbf{b-number} of $\lambda$, denoted by 
$\bnumber(\lambda)$, to be 
$$ \bnumber(\lambda)=\sum_{i}(i-1)\cdots lambda_i . $$  
\end{definition}

\begin{definition}
The \textbf{shifted major index} of $T \in SYT(\lambda)$, denoted by $\barmaj(T)$, 
is $$ \maj(T) + \bnumber(\lambda). $$
The \textbf{shifted comajor index} of $T \in SYT(\lambda)$, denoted by $\barcomaj(T)$, 
is $$ \comaj(T) + \bnumber(\lambda). $$
\end{definition}

\end{comment}

\subsection{Basic facts}
We list those basic facts of (dual-)promotion and (dual-)evacuation that we
will assume. If not specified otherwise, the following facts are about $SYT(\lambda)$
for general $\lambda \vdash n$. 
%These facts were historically developed by Sch\"utzenberger, Edelman, Greene,
%and Haiman. A very detailed account of these facts can be found in Stanley's survey
%paper \cite{RS:2008}. 

\begin{fact} \label{fact: involution}
$\epsilon$ and $\epsilon^*$ are involutions. 
\end{fact}

\begin{fact} \label{fact: dihedral one}
$\epsilon\circ\partial = \partial^*\circ\epsilon$ and 
$\epsilon^*\circ\partial = \partial^*\circ\epsilon^*$. 
\end{fact}

\begin{fact} \label{fact: dihedral two}
$\epsilon\circ\epsilon^* = \partial^n$.
\end{fact}
The above results are due to Sch\"utzenberger~\cite{schutz1, schutz2}. Alternative proofs are given
by Haiman in ~\cite{haiman}.

\begin{fact} \label{fact: promotion on rectangle} 
For any $R \in SYT(c^r)$, let $n=|c^r|=r\cdot c$.  Then $\partial^n(R) = R$. 
\end{fact}
The above result is often attributed to Sch\"utzenberger.

\begin{fact}\label{fact: ev and dev on rect tab}
On rectangular tableaux, $\epsilon$ = $\epsilon^*$. 
%That is, for any $R \in SYT(r^c)$,
%$\epsilon(R)=\epsilon^*(R)$. 
\end{fact}
The above result is an easy consequence of Fact~\ref{fact: dihedral two} and 
Fact~\ref{fact: promotion on rectangle}. 

\begin{fact} \label{fact: promotion on staircase} 
For any $S \in SYT(sc_k)$, let $n=|sc_k|=(k+1)\cdots k/2$.  Then $\partial^{2n}(S) = S$ and
$\partial^n(S) = S^t$, where $S^t$ is the transpose of $S$. 
\end{fact}
The above result is due to Edelman and Greene \cite{e-g}.

\begin{fact} \label{fact: transpose}
For any $S \in SYT(sc_k)$, $\epsilon^*(S) = \epsilon(S)^t$.
\end{fact}
The above result is an easy consequence of Fact~\ref{fact: involution}, 
Fact~\ref{fact: dihedral two}, and Fact~\ref{fact: promotion on staircase}.

\section{The embedding of $SYT(sc_k)$ into $SYT(k^{(k+1)})$} \label{section: the embedding}

In this section we describe the embedding $\iota: SYT(sc_k) \to SYT(k^{(k+1)})$.

\begin{definition}
Given $S \in SYT(sc_k)$, let $N=(k+1)\cdots k$.  Construct $R=\iota(S)$ as follows: 
\begin{itemize}
  \item $R[i,j]=S[i,j]$ for 
        %$i \in [k+1], j \in [k]$ and 
        $i+j \le k+1$ 
        (northwest (upper) staircase portion).
  \item $R[i,j]=N+1-\epsilon(T)[k+2-i,k+1-j]$ for 
        %$i \in [k+1], j \in [k]$ and 
        $i+j > k+1$
        (southeast (lower) staircase portion).
\end{itemize}
\end{definition}

This amounts to the following visualization:
\begin{example}
Let $S=\young(126,35,4)$; then $\epsilon(S)=\young(145,26,3)$. Rotating $\epsilon(S)$  
by $\pi$, we get $\young(::3,:62,541)$. Now we take the complement of each filling by $N+1 =
13$ and get $S'=\young(::\ten,:7\eleven,89\twelv)$. 

There is an obvious way to put $S$ and $S'$ together to create a standard tableau of shape
$3^4$, which is $\iota(t)=\young(126,35\ten,47\eleven,89\twelv)$.    
\end{example}

\begin{remark} \label{remark: equivalent formulation of iota}
Recall that $\epsilon^*(S)=\epsilon(S)^t$. Thus we could have computed
$\epsilon^*(S)=\young(123,46,5)$, and flipped it along the staircase diagonal to get 
$\young(::3,:62,541)$, which is the same as rotating $\epsilon(S)$ by $\pi$. This point of view 
manifests the fact that $n \in \Des(\iota(S))$ (Definition~\ref{definition: extended descent}) if and 
only if the corner of $n$ in $\epsilon^*(S)$ is southeast of the corner of $n$ in $S$.  

It is also an arbitrary choice to embed $SYT(sc_k)$ into $SYT(k^{(k+1)})$ instead of
into $SYT((k+1)^k)$. For example, we could have put together the above $S$ and $S'$ to form 
$$
\young(126\ten,357\eleven,489\twelv).
$$ 

Our arguments below apply to either choice with little modification.
\end{remark}

%The well-definedness and injectivity of above construction is clear. 

From the construction of $\iota$, we see that $\iota(S)$ contains the 
upper staircase portion, which is just $S$, and the lower staircase portion, which 
is essentially $\epsilon(S)$. Therefore, we can just identify $\iota(S)$ 
with the pair $(S, \epsilon(S))$. We would like to understand how the promotion
action on $\iota(S)$ factors through this identification. It is clear from the
construction that promotion on $\iota(S)$, when restricted to the lower staircase
portion, corresponds to dual-promotion on $\epsilon(S)$. If the promotion 
path in $\iota(S)$ passes through the box containing $n = (k+1)\cdots k/2$ (the largest
number in the upper staircase portion of $\iota(S)$), then we know that promotion
on $\iota(S)$, when restricted to the upper staircase portion, corresponds to
promotion on $S$. The following arguments show that this is indeed the case.

%\begin{definition}[\TODO{Should be moved up to the definition of promotion $\partial$}]
%The jeu-de-taquin path corresponding to promotion $\partial$ is called the 
%\underline{promotion path}.  
%The jeu-de-taquin path corresponding to dual-promotion $\partial^*$ is called the 
%\underline{dual-promotion path}.  
%\end{definition}

%\begin{definition}[\TODO{Should be moved elsewhere and expand}]
%Given a rectangular partition $c^r$, the map $\rho : [r]\times[s] \to [r]\times[s]$ given
%by $(i,j) \mapsto (r+1-i, c+1-j)$ is clearly an involution on the set of all boxes in $c^r$.
%$\rho$ can be thought as ``rotating $c^r$ by $\pi$'' or as ``the antipole map''. 

%A partition that is a cross section of the paring of $\rho$ is said to be 
%\underline{doubly covered} by $\rho$. 

%There is a nice bijection between $\{\lambda \mid \lambda \text{ is doubly covered by } c^r\}$ 
%and $\{\text{ 2d random walks of length} \lfloor{(c+r)/2} \rfloor \}$.

%There is a nice bijection between $\{t \in SYT(c^r) \mid \epsilon(t)=t\}$ and 
%$\coprod{\lambda \text{ that is doubly covered by } c^r} SYT(\lambda)$. 
%\end{definition}

\begin{lemma} \label{lemma: key observation}
Let $T \in STY(\lambda)$, and $n=|\lambda|$. If the number $n$ is in box $(i,j)$ of
$T$ (clearly, it must be an outside corner), then the dual-promotion path of $\epsilon^*(T)$ 
ends on box $(i,j)$ of $\epsilon^*(T)$. 
\end{lemma}
\begin{proof}
It follows from the definition of dual-evacuation using dual-sliding that the position
of $n$ in $\epsilon^* \circ \partial(T)$ is the same as the position of $n$ in $T$ 
(because the sliding in the action of promotion and the first application of dual-sliding 
in the definition of dual-evacuation will ``cancel out'' with respect to the position of $n$). 
By the fact that $\epsilon^* (T) =\partial \circ \epsilon^* \circ \partial(T)$ 
(Fact~\ref{fact: dihedral one}) and the fact that the dual-promotion path of 
$\epsilon^*(T)$ is the reverse of the promotion path of $\epsilon^* \circ \partial(T)$
(Remark~\ref{remark: promotion and dual-promotion path}), the statement follows.
\end{proof}

The above lemma, when specialized to staircase-shaped tableaux, implies the following:

\begin{prop} 
Let $S \in SYT(sc_k)$. 
The promotion path of $\iota(S)$ always passes through the box with entry $n=(k+1)\cdots k/2$.
\end{prop}
\begin{proof}
Suppose $n$ is in box $(i,j)$ of $S \in SYT(sc_k)$.  
Since $S$ is of staircase shape, we have $\epsilon^*(S) = \epsilon(S)^t$ (Fact~\ref{fact: transpose}). 
The above lemma then says the dual-promotion path of $\epsilon(S)$ ends on box $(j,i)$
of $\epsilon(S)$, which is ``glued'' exactly below box $(i,j)$ of $S$ by the 
construction of $\iota$. Now we use the observation that the promotion path of $\iota(S)$,
when restricted to the lower staircase portion, corresponds to the dual-promotion
path of $\epsilon(S)$.  The result follows.  

%$\epsilon \circ \epsilon^*(S) = \partial^n(S) = S^t$.  
%By Lemma \ref{lemma: key observation}, the dual-promotion path of $\epsilon^*(t)$ ends on
%box $(j,i)$.  If the dual-promotion path of $\epsilon^*(t)$ is $[(m_1,n_1),(m_2,n_2),\ldots]$ then
% the beginning of the promotion path on $\iota(S)$ is given by
%$[(k+2-m_1,k+1-n_1),(k+2-m_2,k+1-n_1),\ldots]$.  Since $(i,j)$ is an outside corner of $sc_k$, we
%have $i + j = k+1$, so that $(i+1,j)$ occurs in the promotion path on $\iota(S)$.  Therefore, $n$
%must move via jeu-de-taquin from box $(i,j)$ in $\iota(S)$ to box $(i+1,j)$, so
%$n$ is in the promotion path of $\iota(S)$.
\end{proof}

This proves our first main result of the embedding $\iota$. 

\begin{theorem} \label{theorem: embedding preserves promotion}
For $S \in SYT(sc_k)$, $\iota\circ\partial(S) = \partial\circ\iota(S)$.
\end{theorem}
%\begin{proof}
%If $\pi$ is rotation by $\pi$, and $c$ is complementation by $N+1$, 
%then we have 
%$$\partial \circ c \circ \pi \circ \epsilon^*(S) = 
%c \circ \pi \circ \partial \circ \epsilon^*(S) = c \circ \pi \circ \epsilon^* \circ \partial(S).$$  
%The theorem follows.
%\end{proof}

%It is also easy to see from the construction that $\iota$ preserves the descent vector.
%\begin{prop} \label{prop: embedding preserves dv}
%For $S \in SYT(sc_k)$, $\extDes(S)=\extDes(\iota(S))$. 
%\end{prop}

%By Remark~\ref{remark: edeodv} and Prop~\ref{prop: embedding preserves dv}, it is tempting
%to hope that $\iota\circ\epsilon^* = \epsilon\circ\iota$. It is indeed the case:

By the above theorem and the definition of evacuation, we have that 
\begin{theorem} \label{theorem: embedding preserves evacuation}
For $S \in SYT(sc_k)$, $\iota\circ\epsilon(S) = \epsilon\circ\iota(S)$.
\end{theorem}
%\begin{proof}
%Use that $\epsilon^*$ can alternatively be defined where you don't modify the shape.  
%That is, define dual evacuation as follows.  Have an empty tableau $S$.  
%At step $k$, the smallest entry of $\partial^k(T)$ that is not yet filled in in $S$ 
%you fill in with a $k$.  The end result is $\epsilon^*(S)$.  
%Clearly, $\iota \circ \epsilon^* = \epsilon^* \circ \iota$ on the ``top half'' of $\iota(S)$.  
%The bottom half of $\iota \circ \epsilon^*(S)$ is just $c \circ \pi(S)$.  
%The bottom half of $\epsilon^* \circ \iota(S)$ is 
%$$\epsilon \circ c \circ \pi \circ \epsilon^*(S) = 
%c \circ \pi \circ \epsilon^* \circ \epsilon^*(S) = c \circ \pi (S).$$
%\end{proof}

\begin{remark}
It can be show either independently or as a corollary of 
Theorem~\ref{theorem: embedding preserves promotion} that
$$\iota\circ\partial^*(S) = \partial^*\circ\iota(S).$$

On the other hand, it is \emph{not} true that $\iota\circ\epsilon^*(S) = \epsilon^*\circ\iota(S)$.
On the contrary by Fact~\ref{fact: ev and dev on rect tab} we know that
$$\iota\circ\epsilon(S) = \epsilon^*\circ\iota(S).$$
It is not hard to see that
$$\iota\circ\epsilon^*(S) = \epsilon\circ\iota(S^t).$$
\end{remark}

\section{Descent vector} \label{section: descent vector}

\subsection{Descent vector of rectangular tableaux}
Rhoades \cite{rhoades} invented the notion of ``extended descent'' in order to describe 
the promotion action on rectangular tableaux: 
\begin{definition} \label{definition: extended descent} 
Let $R \in SYT(r^c)$, and $n=c\cdots r$. We say $i$ is an \textbf{extended descent} of $R$ if 
either $i$ is a descent of $R$, or $i=n$ and $1$ is a descent of $\partial(R)$. 
The \textbf{extended descent set} of $R$, denoted by $\extDes(R)$, is the set of all
extended descents of $R$. 
\end{definition}

\begin{example}
In the case that $R_1=\young(136,257,49\eleven,8\ten\twelv)$, $\extDes(R_1)=\{1,3,6,7,9,11\}$.
Here $12 \not\in \extDes(R_1)$ because $1$ is not a descent of 
$\partial(R_1)=\young(127,348,56\ten,9\eleven\twelv)$. 

In the case that $R_2=\young(124,359,68\eleven,7\ten\twelv)$, $\extDes(R_2)=\{2,4,5,6,9,11,12\}$. 
Here $12 \in \extDes(R_2)$ because $1$ is a descent of 
$\partial(R_2)=\young(135,246,79\ten,8\eleven\twelv)$.
\end{example}

It is often convenient to think of $\extDes(R)$ as an array of $n$ boxes, where a dot is put at 
the $i$-th box of this array if and only if $i$ is an extended descent of $R$. 
In this form, we will call 
$\extDes(R)$ the \textbf{descent vector} of $R$. Furthermore, we identify 
(``glue together'') the left edge of the left-most box and the right edge of the right-most box so that
 the array $\extDes(R)$ forms a circle.
It therefore makes sense to talk about rotating $\extDes(R)$ to the right, where the 
content of the $i$-th box goes to the $(i+1)$-st box (mod n), or similarly, rotating to the left.    

\begin{example}
Continuing the above example, 
$$\extDes(R_1)=\young(\bullet\hfil\bullet\hfil\hfil\bullet\bullet\hfil\bullet\hfil\bullet\hfil)$$ 
and 
$$\extDes(R_2)=\young(\hfil\bullet\hfil\bullet\bullet\bullet\hfil\hfil\bullet\hfil\bullet\bullet) .$$ 
\end{example}

We would like to point out that the map $\extDes : SYT(r^c) \to (0,1)^n$ is not injective and that
the pre-images of $\extDes$ are not equinumerous in general. 

Rhoades \cite{rhoades} showed a nice property of the promotion action on the extended descent set. 
In the language of descent vectors, it has the following visualization:

\begin{theorem}[Rhoades, \cite{rhoades}] \label{theorem: pordv}
If $R$ is a standard tableau of rectangular shape, then the promotion $\partial$ 
rotates $\extDes(R)$ to the right by one position. 
\end{theorem}

\begin{example}
Continuing the above example, if $R_3=\partial(R_2)=\young(135,246,79\ten,8\eleven\twelv)$ 
then 
$$\extDes(R_3)=
\young(\bullet\hfil\bullet\hfil\bullet\bullet\bullet\hfil\hfil\bullet\hfil\bullet) .$$ 
\end{example}

The action of evacuation $\epsilon$ on descent vectors is also very nice:
(Note that dual-evacuation $\epsilon^*$ is the same as evacuation $\epsilon$ on rectangular
tableaux.)

\begin{theorem} \label{theorem: eordv}
Let $R \in SYT(r^c)$ and $n=c\cdots r$.  Then evacuation $\epsilon$ 
rotates $\extDes(R)$ to the right by one position and then flips the result of the 
rotation. More precisely, the $i$-th box of $\extDes(\epsilon(R))$ is dotted
if and only if the $(n-i)$-th (mod n) box of $\extDes(R)$ is dotted.  
\end{theorem} 
\begin{proof}
We first note that $\epsilon(R)=\epsilon^*(R)$ (Fact~\ref{fact: ev and dev on rect tab}).
Then we note that $\Des(R)$ is the set of left descents of the column reading word $w_R$ 
of $R$ (Remark~\ref{remark: reading word descent}).  
Now, $i$ is a left descent of $w_R$ if and only if $n-i$ is a left descent of $w_R^\sharp$
(Remark~\ref{remark: RSK evacuation}). Therefore $i\in \Des(R)$ if and only if 
$n-i \in \Des(\epsilon(R))$.  

If $n \in \extDes(R)$, then $1 \in \Des(\partial(R))$ 
(Definition~\ref{definition: extended descent}), 
thus $n-1 \in \Des(\epsilon\circ\partial(R))$ by the previous paragraph,
thus $n-1 \in \Des(\partial^{-1}\circ\epsilon(R))$ (Fact~\ref{fact: dihedral one}), 
thus $n \in \extDes(\epsilon(R))$ (Theorem~\ref{theorem: pordv}). 
Since $\epsilon$ is an involution, the converse is also true.  
\end{proof}

\begin{example}
Continuing the above example, $\epsilon(R_3)=\young(125,346,79\eleven,8\ten\twelv)$ 
and 
$$\extDes(\epsilon(R_3))=
\young(\hfil\bullet\hfil\hfil\bullet\bullet\bullet\hfil\bullet\hfil\bullet\bullet) .$$ 
It is clear that this action is an involution.
\end{example}

\subsection{Descent vector of staircase tableaux}
For staircase tableaux, we give the following construction of descent vector.
\begin{definition}
Let $S \in SYT(sc_k)$ and $n=|sc_k|=(k+1)\cdots k/2$. Then $\extDes(S)$ is an array of $2n$ boxes. The rules
of placing dots into these boxes are the following.
\begin{itemize} 
   \item If $i \in Des(S)$, then put a dot in the $i$-th box and leave the $(n+i)$-th box empty.
   \item If $i \not\in Des(S)$, then put a dot in the $(n+i)$-th box and leave the $i$-th box empty.
   \item If $1 \in Des(\partial(S))$, then leave the $n$-th box empty and put a dot in the 
         $(2n)$-th box. 
   \item If $1 \not \in Des(\partial(S))$, then leave the $(2n)$-th box empty and put a dot in the 
         $n$-th box. 
\end{itemize} 
We identify the left edge and the right edge of this array. 
\end{definition}

\begin{example}
In the case that $S_1=\young(145,26,3)$, 
$$\extDes(S_1)=\young(\bullet\bullet\hfil\hfil\bullet\bullet\hfil\hfil\bullet\bullet\hfil\hfil).$$ 
In the case that $S_2=\young(125,36,4)$, 
$$\extDes(S_2)=\young(\hfil\bullet\bullet\hfil\bullet\hfil\bullet\hfil\hfil\bullet\hfil\bullet).$$ 
\end{example}

As in the case of rectangular tableaux, the map $\extDes$ is not injective and the pre-images of 
$\extDes$ are not equinumerous in general. 

%\begin{remark} \label{remark: first and second half}
From the definition, we see that the first half and the second half of $\extDes(S)$ are 
just complements of each other, that is, for each $i \in [(k+1)\cdots k]$ precisely one of the 
$i$-th and $((k+1)\cdots k+i)$-th boxes is dotted. Thus the second half of $\extDes(S)$ is
redundant. On the other hand, this redundancy demonstrates the link between $\extDes(S)$
and $\extDes(\iota(S))$ as stated in Theorem \ref{theorem: edv and iota}.  First, we need a
supporting lemma, whose proof is not hard but rather tedious, so we leave it to the
appendix. 
%\end{remark}

\begin{lemma} \label{lemma: promotion path and dual-promotion path} 
Let $S \in SYT(sc_k)$ and $n=|sc_k|$. 
If the promotion path of $S$ ends with a vertical (up) move, 
then the corner of $n$ in $\epsilon^*(S)$ is northeast of the corner of $n$ in $S$.   
%then the whole dual-promotion path of $T$ must be (weakly) northeast
%of the promotion path. 
If the promotion path of $S$ ends with a horizontal (left) move, 
then the corner of $n$ in $\epsilon^*(S)$ is southwest of the corner of $n$ in $S$.   

%then the whole 
%dual-promotion path of $T$ must be (weakly) southwest of the promotion path.
\end{lemma}

\begin{theorem} \label{theorem: edv and iota}
For $S \in SYT(sc_k)$, $\extDes(S)=\extDes(\iota(S))$.
\end{theorem}
\begin{proof}
Parsing through the construction of $\iota$, we see that this claim is the conjunction
of the following two statements:
\begin{enumerate} 
  \item for $i \not = n$, $i \in \Des(S)$ if and only if $n-i \not\in \Des(\epsilon(S))$; and
  \item $1 \in \Des(\partial(S))$ if and only if $n \not \in \Des(\iota(S))$.
\end{enumerate} 
For the first statement, we note that $i \in \Des(S)$ is equivalent to that $i$ is a left 
descent of a reading word $w_S$ of $S$ (Remark~\ref{remark: reading word descent}), which
is equivalent to that $n-i$ is a left descent of the word $w_s^\sharp$ 
(Remark~\ref{remark: RSK evacuation}), which is equivalent to that $n-i$ is a descent
in $\epsilon^*(S)$, which is equivalent to that $n-i$ is not a descent of $\epsilon(S)$ 
(Fact~\ref{fact: transpose}).

Now, $1 \in \Des(\partial(S))$ is equivalent to that the promotion path of $S$ ends with
a vertical (up) move, which is equivalent to that the corner $n$ in $\epsilon^*(S)$ is
northeast of the corner of $n$ in $S$ by Lemma \ref{lemma: promotion path and dual-promotion path}
, which is equivalent to that $n \not \in \iota(S)$
(Remark~\ref{remark: equivalent formulation of iota}).

\end{proof}

%There are several easy facts about $\extDes(S)$ for staircase tableau $S$, they follow directly
%from the definition: 

%\begin{facts}
%Let us denote the transpose of $S$ by $S^t$, which clearly is still a standard tableau of the
%same staircase shape as $S$, then
%\begin{itemize}
%   \item the number of dots is equal to the number of empty boxes and both are 
%         equal to $n$, 
%   \item $\extDes(S^t)$ is the complement of $\extDes(S)$, (that is the $i$-th position of
%         $\extDes(S^t)$ is dotted if and only if the $i$-th position of $\extDes(S)$ is empty.) 
%\end{itemize}
%\end{facts}

The above Theorems~\ref{theorem: edv and iota},  ~\ref{theorem: embedding preserves promotion} and  
~\ref{theorem: pordv} imply the following analogy to Theorem~\ref{theorem: pordv}  
for staircase tableaux:
\begin{theorem} \label{theorem: posdv}
If $S$ is a standard tableau of staircase shape, then
promotion $\partial$ rotates $\extDes(S)$ to the right for one position. 
\end{theorem}
%\begin{proof}
%This follows from the embedding $\iota : SYT(sc_k) \to SYT((k+1)^k)$ 
%constructed in Section \ref{section: the embedding}.
%\end{proof}

Note that if we rotate $\extDes(S)$ in any direction by $n$ positions we get the complement
of $\extDes(S)$, which is $\extDes(S^t)$. This agrees with Edelman and Greene's result \cite{e-g} that
$\partial^n(S)=S^t$ and $\partial^n=\partial^{-n}$.

Unlike the case of rectangular tableaux, evacuation $\epsilon$ and dual-evacuation $\epsilon^*$
act differently on staircase tableaux. Their actions on descent vectors are described below:  

\begin{theorem} \label{theorem: eosdv}
Let $S \in SYT(sc_k)$ and $n=(k+1)\cdots k/2$.  Then evacuation $\epsilon$ 
rotates $\extDes(S)$ to the right by one position and then flips the result of the 
rotation. More precisely, the $i$-th box of $\extDes(\epsilon(S))$ is dotted
if and only if the $(2\cdots (n-i))$-th box of $\extDes(S)$ is dotted.  
\end{theorem}

\begin{theorem} \label{theorem: deosdv}
Let $S \in SYT(sc_k)$ and $n=(k+1)\cdots k/2$.  Then dual-evacuation $\epsilon^*$ 
rotates $\extDes(S)$ to the right by $n-1$ position and then flips the result of the 
rotation. More precisely, the $i$-th box of $\extDes(\epsilon(S))$ is dotted
if and only if the $(n-i)$-th box of $\extDes(S)$ is dotted.  
\end{theorem}

\begin{example}
Let $S_3=\young(124,36,5)$, then 
    $\epsilon(S_3)=\young(135,24,6)$ and
    $\epsilon^*(S_3)=\young(126,34,5)$. Correspondingly,
$$\extDes(S_3)=\young(\hfil\bullet\hfil\bullet\hfil\hfil\bullet\hfil\bullet\hfil\bullet\bullet),$$ 
$$\extDes(\epsilon(S_3))=
   \young(\bullet\hfil\bullet\hfil\bullet\hfil\hfil\bullet\hfil\bullet\hfil\bullet),$$ 
and
$$\extDes(\epsilon^*(S_3))=
   \young(\hfill\bullet\hfill\bullet\hfil\bullet\bullet\hfil\bullet\hfil\bullet\hfil).$$ 

Note that $\extDes(\epsilon(S_3))$ is the complement of $\extDes(\epsilon^*(S_3))$.
This agrees with the fact that $\epsilon^*(S_3) = \epsilon(S_3)^t$. 
\end{example}

%\begin{remark} \label{remark: edeodv}
%From the definition we see that the action of dual-evacuation $\epsilon^*$ on the descent vector
%of rectangular tableau is the same as the dual-evacuation $\epsilon^*$ on the descent vector 
%of staircase tableau. 
%\end{remark}

\begin{proof}[Proof of Theorem~\ref{theorem: eosdv} and \ref{theorem: deosdv}]
Theorem~\ref{theorem: eosdv} follows directly from Theorem~\ref{theorem: edv and iota} and
Theorem~\ref{theorem: embedding preserves evacuation}. 

Theorem~\ref{theorem: deosdv} follows from the fact that $\extDes(S)$ is the complement of
$\extDes(S^t)$.
\end{proof}

Theorems~\ref{theorem: pordv} and \ref{theorem: posdv} imply that if $T$ is either a
rectangular or staircase tableau, $\extDes(T)$ encodes important information about the promotion 
cycle that $T$ is in.

\begin{corollary}
If $T$, either of rectangular or staircase shape, is in a promotion cycle of size $C$ 
then $\extDes(T)$ must be periodic with period dividing $C$. (The period does not have to 
be exactly $C$.)
\end{corollary}

\begin{example}
Let $T=\young(159,26\ten,37\eleven,48\twelv)$, then 
$$\extDes(T)=\young(\bullet\bullet\bullet\hfil\bullet\bullet\bullet\hfil\bullet\bullet\bullet\hfil).$$
We see that $\extDes(T)$ has a period of 4, thus $T$ must be in a promotion cycle of size either
4 or 12. Indeed, the promotion order of $T$ is 4.

On the other hand, the promotion order of $T=\young(135,279,48\eleven,6\ten\twelv)$ is
also 4, while its descent vector 
$$\extDes(T)=\young(\bullet\hfil\bullet\hfil\bullet\hfil\bullet\hfil\bullet\hfil\bullet\hfil)$$
has period 2.
\end{example}

Equipped with the above knowledge, we can say more about the promotion action on $SYT(sc_k)$.
For example: 

\begin{corollary} \label{corollary: full cycle}
In the promotion action on $SYT(sc_k)$
there always exists a full cycle, that is, a cycle of the same size as
the order of the promotion $\partial$, in this case $(k+1)\cdots k$. 
\end{corollary} 
\begin{proof}
Consider $T \in SYT(sc_k)$ obtained by filling the numbers $1$ to $(k+1)\cdots k/2$ down columns,
from leftmost column to rightmost column. Then $\extDes(T)$ has period of $(k+1)\cdots k$, thus 
$T$ must be in a full cycle.  

\end{proof}

For example for $k=3$, and $T=\young(146,25,3)$,
$$\extDes(T)=\young(\bullet\bullet\hfil\bullet\hfil\bullet\hfil\hfil\bullet\hfil\bullet\hfil)$$
has period 12.

Indeed, computer experiments show that ``most'' cycles of the promotion action are
full cycles.

\begin{corollary} \label{corollary: no half cycle}
In the promotion action on $SYT(sc_k)$,
let $N=(k+1)\cdots k$. If a cycle of length $C$ appears, then $C$ is a divisor of $N$, but
%   $$ N/2 \equiv C/2 \text{ (mod $C$)}. $$
not a divisor of $N/2$.
\end{corollary}
\begin{proof}
The cycle size $C$ is a divisor of $N$ since the order of promotion $|\partial|$ is $N$.

On the other hand, $C$ cannot be a divisor of $N/2$ since by definition $\extDes(T)$ can never have period of
length that is a divisor of $N/2$. 
\end{proof}

\section{Some comments and questions} \label{section: future}
The discovery of $\iota$ is a by-product of our attempt to solve an open question
posed by Stanley (\cite[page 13]{RS:2008}) that asks if Rhoades' CSP result
on rectangular tableaux can be extended to other shapes, and if there is a 
more combinatorial proof of this result. 

Rhoades' proof uses Kazhdan-Lusztig theory, requiring special properties of
rectangular tableaux. It seems (to us) that there is not an obvious analogous proof
for other shapes.  
%At the same time, we were inspired by the success of Nicolas M. Thi\'ery 
%and Anne Schilling in generating and testing conjectures 

So we decided to try our luck in computer exploration using Sage-Combinat 
(\cite{sage}, \cite{Sage-Combinat}). The first thing we noticed from the computer data 
was the nice promotion cycle structure of staircase
tableaux, which is not a surprise at all due to Fact~\ref{fact: promotion on staircase}.
Thus we decided to focus on Problem~\ref{problem 2}.  

It was soon clear to us that brute-force computation of the cycle structure could not proceed
very far; we could only handle $SYT(sc_k)$ for $k \le 5$ on our computer. 
On the other hand, the promotion cycle 
structures on rectangular tableaux are extremely easy to compute by Rhoades' result,
as the generating function of $\overline{\maj}$ is just the $q-$analogue of the hook 
length formula. So the embedding $\iota$ is an effort to study the promotion cycle 
structure on $SYT(sc_k)$ by borrowing information from the promotion action on $SYT(k^{k+1})$. 

Among the cases of promotion action on $SYT(sc_k)$ for which we know the complete cycle
structure (that is, $k = 3,4,5$), we have found that each has a CSP polynomial that is
a product of cyclotomic polynomials of degree $\le (k+1)\cdots k$: For $k=3,4,5$, these
polynomials are 

$$\Phi_2\Phi_4^2\Phi_6\Phi_8\Phi_{12},$$
$$\Phi_2^3\Phi_3\Phi_4^2\Phi_8\Phi_{10}^2\Phi_{16}\Phi_{20}, \text{ and}$$
$$\Phi_2^{11}\Phi_6\Phi_{10}^3\Phi_{11}\Phi_{13}\Phi_{22}\Phi_{24}^4\Phi_{30},$$
respectively.

We note that these polynomials in product form are not unique, for example 
$$\Phi_2^2\Phi_4\Phi_6\Phi_{10}\Phi_{12}$$ gives another CSP polynomial for 
$SYT(sc_3)$. 

The study of this product form continues, with the hope of finding a counting formula
the $q-$analogue of which is a CSP polynomial for the promotion action on $SYT(sc_k)$.  

For the case $k>5$, Corollary~\ref{corollary: no half cycle} gives a necessary
condition for what kind of cycles can appear in the promotion action on $SYT(sc_k)$.
We do not know how sufficient this condition is. 

The question that interests us the most is if the embedding has any 
representation-theoretical interpretation, and if such an interpretation can 
help settle Problem~\ref{problem 2}. 

\section*{acknowledgements}
The authors want to thank Anne Schilling for her scholarly and financial support of 
this work; the initial idea that such an embedding may exist was first suggested to us by
Anne. We want to thank Vic Reiner for providing references and commenting on an earlier version 
of this paper. We also would like to thank Andrew Berget and Richard Stanley for 
discussion on this topic, and Nicolas M. Thi\'ery for his 
inspiration with regards to computer exploration. 

\appendix
\section{Proof of Lemma~\ref{lemma: promotion path and dual-promotion path}} 

To prove Lemma~\ref{lemma: promotion path and dual-promotion path}, we first make
the observation that the location of the corner that contains $n$ (the $n$-corner) in $S$ cannot 
be the same as the $n$-corner in $\epsilon^*(S)$. This is because, as we had observed in 
Lemma~\ref{lemma: key observation}, the $n$-corner in $S$ is the same as the
$n$-corner in $\epsilon^*\circ\partial(S)$; and 
$\epsilon^*(S)=\partial\circ\epsilon^*\circ\partial(S)$ does not have the same
$n$-corner as that of $\epsilon^*\circ\partial(S)$.

With this observation, Lemma~\ref{lemma: promotion path and dual-promotion path}
is a consequence of the following general fact:

\begin{lemma} 
Let $T \in SYT(\lambda)$ and $\lambda \vdash n$. 
If the promotion path of $T$ ends with a vertical (up) move, 
then the whole dual-promotion path of $T$ must be (weakly) northeast of the promotion path. 

If the promotion path of $T$ ends with a horizontal (left) move, 
then the whole dual-promotion path of $T$ must be (weakly) southwest of the promotion path.
\end{lemma} 
\begin{proof}
Without loss of generality, we argue the case where the promotion path of $T$ ends with
a vertical move.

Imagine a boy and a girl standing at the most northwest box of $T$. The boy will walk
along the promotion path in reverse towards the southeast, and the girl will
walk along the dual-promotion path towards the south-east. They will walk at the same speed.

In the first step, the boy goes south by assumption. The girl may go east or south.
If she starts by going east, then she is already strictly northeast of the boy. If she 
starts by going south with the boy, then she must turn east earlier than the boy turns. 
(Suppose the boy turns east at box $(i,j)$.  By definition of promotion path, this implies that 
$T[i,j] > T[i-1,j+1]$.  If the girl goes south at box $(i-1,j)$ then by definition of dual-promotion
path, this implies that $T[i,j] < T[i-1,j+1]$, a contradiction.  Therefore, the girl must turn east at
box $(i-1,j)$ or earlier.)  So either way we see that the girl will be strictly 
northeast of the boy before the boy makes his first east turn.

If they never meet again then we are done. So we assume that their next meeting 
position is at the box $(s,t)$, and argue that they will never cross. By induction 
this will prove the claim.

It is clear that the girl must enter the box $(s,t)$ from north, and the boy must enter
the box $(s,t)$ from the west. From box $(s,t)$, the girl can either go south or go east.

Suppose the girl goes south from $(s,t)$. Then the boy must also go south from $(s,t)$. 
(For if he went east, it would imply that $T[s-1,t+1]<T[s,t]$, which would make the
girl go through $(s-1,t+1)$ instead of $(s,t)$.) Then we can use our previous argument to
show that the girl must make an east turn before the boy, and stay northeast of the boy.

Suppose the girl goes east from $(s,t)$. Then again the boy must go south from $(s,t)$.
(For the girl's behaviour shows that $T[s-1,t+1]>T[s,t]$, but the boy's going east   
would imply that $T[s-1,t+1]<T[s,t]$.) So the girl stays northeast of the boy.
\end{proof}

\end{document}